\documentclass[a4paper, 12pt, twoside]{article}

\usepackage[a4paper, margin=1in]{geometry}
\usepackage[latin1]{inputenc}
\usepackage[T1]{fontenc}
\usepackage[english]{babel}
\usepackage{textcomp} 
\usepackage{graphicx}
\DeclareGraphicsExtensions{.jpg,.mps,.pdf,.png,.gif}
\usepackage[french]{varioref} 
\usepackage[]{amsmath,amssymb,amsfonts}

\usepackage{float}
\usepackage{bbm}
\usepackage{bbold}
\usepackage{amsthm}
\usepackage{dsfont}
\usepackage{color}
\usepackage{natbib}

\usepackage[hyperindex,breaklinks]{hyperref}

\newcommand{\Ben}{\begin{enumerate}}
\newcommand{\Een}{\end{enumerate}}
\newcommand{\Bit}{\begin{itemize}}
\newcommand{\Eit}{\end{itemize}}
\newcommand{\Beq}{\begin{equation}}
\newcommand{\Eeq}{\end{equation}}
\newcommand{\Ba}{\begin{align*}}
\newcommand{\Ea}{\end{align*}}
\newcommand{\Mb}{\mathbf}

\newtheorem{Th}{Theorem}

\newtheorem{Rq}{Remark} 

\newtheorem{Def}{Definition}

\title{Geometric ergodicity for some space-time max-stable Markov chains}
\date{May 8, 2018}
\begin{document}
\author{
~~ Erwan Koch\footnote{EPFL (Chair of Statistics STAT): EPFL-SB-MATH-STAT, MA B1 433 (B\^atiment MA), Station 8, 1015 Lausanne, Switzerland. ETH Zurich (Department of Mathematics, RiskLab). \newline Email: erwan.koch@epfl.ch}
~~ Christian Y. Robert\footnote{ISFA Universit\'e Lyon 1, 50 Avenue Tony Garnier, 69366 Lyon cedex 07, France. \newline Email: christian.robert@univ-lyon1.fr}
}

\maketitle

\begin{abstract}
Max-stable processes are central models for spatial extremes. In this paper, we focus on some space-time max-stable models introduced in \cite{embrechts2016space}. The processes considered induce discrete-time Markov chains taking values in the space of continuous functions from the unit sphere of $\mathds{R}^3$ to $(0, \infty)$. We show that these Markov chains are geometrically ergodic. An interesting feature lies in the fact that the state space is not locally compact, making the classical methodology inapplicable. Instead, we use the fact that the state space is Polish and apply results presented in \cite{hairer2010p}.

\medskip

\end{abstract}

\section{Introduction}

Max-stable processes constitute an extension of multivariate extreme-value theory to the level of stochastic processes \citep[see, e.g.,][]{haan1984spectral, de2007extreme} and turn out to be fundamental for the modelling of spatial extremes. In the related literature, measurements are often assumed to be independent in time and, thus, only the spatial structure is studied \citep[see, e.g.,][]{padoan2010likelihood}. To the best of our knowledge, only \cite{davis2013max}, \cite{huser2014space}, \cite{buhl2015anisotropic} and \cite{embrechts2016space} propose space-time max-stable processes. The class of models introduced in \cite{embrechts2016space}, i.e. the space-time max-stable models with spectral separability, allows to overcome some drawbacks inherent to the approach taken in the other mentioned papers; see \cite{embrechts2016space} for details. 

In this study, we focus on a generalised version of the subclass of ``models of type 2'' defined in \cite{embrechts2016space}, Section 2.1.2. One remarkable feature of the associated models is to be space-time max-stable models on the unit sphere of $\mathds{R}^3$. Although max-stable processes on a sphere have, to the best of our knowledge, only been considered in \cite{embrechts2016space}, such processes can be relevant for applications due to the natural spherical shape of planets and stars. As will be seen, in the discrete-time case, some of the models mentioned directly above induce Markov chains taking values in the space of continuous functions from the unit sphere of $\mathds{R}^3$ to $(0, \infty)$. For an excellent review of Markov chains theory, we refer the reader to \cite{meyn2009markov}. The main result of the present paper is the geometric ergodicity of the just mentioned Markov chains. Since the state space is not locally compact, geometric ergodicity cannot be obtained using classical results, contained, e.g., in \cite{meyn2009markov}. Instead, we take advantage of the fact that the state space is Polish and apply results for Markov chains with Polish state spaces to be found in \cite{hairer2010p}. Conditions for geometric ergodicity of Markov chains with Polish state spaces have been barely considered in the literature so far. Geometric ergodicity is a very powerful property: under some specific moment condition, any transformation of a geometrically ergodic Markov chain satisfies a central limit theorem; see, e.g., \cite{ibragimov1971independent}, Chapter 19, Section 1, for uniformly geometrically ergodic (i.e. uniformly ergodic) Markov chains and \cite{haggstrom2005central}, especially Theorem 1.2, for geometrically ergodic Markov chains. Moreover, \cite{miasojedow2014hoeffding} has shown that some functions of $\pi$-irreducible reversible geometrically ergodic Markov chains with Polish state spaces satisfy an inequality of Hoeffding type; see Remark 3.2 and Theorem 3.3. Hence, geometric ergodicity allows to obtain results for statistics based on the Markov chain and to carry out statistical inference.

The remainder of the paper is organised as follows. Section \ref{Sec_Space_Time_Max_Stable_Sphere} describes the previously mentioned class of ``models of type 2'' as well as its generalised version. Then, our Markov chains are presented and their geometric ergodicity is shown in Section \ref{Sec_Geometric_Ergodicity}. Finally, Section \ref{Sec_Conclusion} provides a short summary as well as some perspectives. Throughout the paper, ``$\bigvee$'' denotes the supremum when the latter is taken over a countable set. Additionally, $\overset{d}{=}$ stands for equality in distribution. In the case of stochastic processes, this must be understood as equality of finite-dimensional distributions.

\section{The subclass of space-time max-stable models of type 2 and its generalisation}
\label{Sec_Space_Time_Max_Stable_Sphere}

First, we recall the definition of the subclass composed of the ``models of type 2'' specified in \cite{embrechts2016space}, Section 2.1.2. Before doing so, we need to introduce some notations and concepts. Let $\mathcal{T}$ be the set of time indices. The mentioned models are either continuous-time ($\mathcal{T}=\mathds{R}$) or discrete-time ($\mathcal{T}=\mathds{Z}$). Let $\lambda$ be the Lebesgue measure on $\mathds{R}$ (case $\mathcal{T}=\mathds{R}$) or the counting measure $\sum_{z\in \mathds{Z}}\delta_{\{z\}}$, where $\delta$ stands for the Dirac measure (case $\mathcal{T}=\mathds{Z}$). Denoting by $\|.\|$ the Euclidean norm, we introduce $\mathds{S}^{2}= \{ \Mb{x} \in \mathds{R}^3: \| \Mb{x} \|=1 \}$, the unit sphere in $\mathds{R}^3$. Moreover, let $\lambda_{\mathds{S}^2}$ be the Lebesgue measure on $\mathds{S}^{2}$. Let $f$ be the von Mises-Fisher probability density function \citep[see, e.g.,][Section 9.3.2]{mardia2009directional} on $\mathds{S}^{2}$ with
parameters $\boldsymbol{\mu }\in \mathds{S}^{2}$ and $\kappa \geq 0$:
\begin{equation}
\label{Eq_von_Mises_Fisher}
f(\mathbf{x};\boldsymbol{\mu },\kappa )=\frac{\kappa }{4\pi \sinh ( \kappa ) }\exp
\left( \kappa \boldsymbol{\mu }' \mathbf{x}\right) \text{,}\quad 
\mathbf{x}\in \mathds{S}^{2},
\end{equation}
where $\sinh$ designates the hyperbolic sine function and $'$ denotes transposition.
The parameters $\boldsymbol{\mu }$ and $\kappa $ are termed the mean direction
and concentration parameter, respectively. The higher the value of $\kappa $, the greater the concentration of the distribution around the mean direction 
$\boldsymbol{\mu }$. The distribution is uniform on the sphere for $\kappa =0$ and unimodal for $\kappa >0$. In addition, for $\mathbf{u}=(u_{x},u_{y},u_{z})' \in $ $\mathds{S}^{2}$, let $R_{\theta ,\mathbf{u}}$ be the rotation matrix of angle $\theta \in \mathds{R}$
around an axis in the direction of $\mathbf{u}$. The latter is written
$$
R_{\theta ,\mathbf{u}}=\cos \theta I_{3}+\sin \theta \lbrack \mathbf{u}
]_{\times }+(1-\cos \theta ) \Mb{u} \Mb{u}',
$$
where $I_{3}$ is the identity matrix of $\mathds{R}^{3}$ and $[\mathbf{u}]_{\times }$ the cross product matrix of $\mathbf{u}$, defined by
$$
\lbrack \mathbf{u}]_{\times }=
\begin{pmatrix}
0 & -u_{z} & u_{y} \\ 
u_{z} & 0 & -u_{x} \\ 
-u_{y} & u_{x} & 0
\end{pmatrix}.
$$
Furthermore, let $g$ be a probability density function (case $\mathcal{T}=\mathds{R}$) or a probability mass function (case $\mathcal{T}=\mathds{Z}$).
Finally, we recall the way we define a Poisson point process on $\mathds{Z}$ in \cite{embrechts2016space}. Let $\left (
N_{k}\right ) _{k\in \mathds{Z}}$ be independent and identically distributed Poisson$(1)$, where, for $\lambda_p>0$, Poisson$(\lambda_p)$ stands for the Poisson distribution with parameter $\lambda_p$. The process $N$ defined by $N\left( A\right) =\sum_{k\in A}N_{k}$, $A\subset \mathds{Z}$, is a Poisson point process on $\mathds{Z}$ with intensity one. Indeed, the quantity $N(A)$ is Poisson distributed with parameter $\lambda(A)$ and, for any $l \geq 1$ and $A_1, \dots, A_l$ disjoint sets in $\mathds{Z}$, the $N(A_i)$, $i=1, \dots, l,$ are independent random variables.
The definition of the ``models of type 2'' introduced in \cite{embrechts2016space} is recalled immediately below.
\begin{Def}
The ``models of type 2'' in \cite{embrechts2016space} are defined by
\begin{equation}
\left( X(t,\Mb{x}) \right)_{(t,\Mb{x})\in \mathcal{T}\times \mathds{S}^{2}}=\left( \bigvee_{i=1}^{\infty } \{ U_{i} g(t-B_i) f \left( R_{\theta (t-B_i),\mathbf{u}} \mathbf{x}; \boldsymbol{\mu}_i, \kappa \right) \} \right)_{(t,\Mb{x})\in \mathcal{T}\times \mathds{S}^{2}},
\label{SpectralRepresentation}
\end{equation}
where $(U_{i},B_{i},\boldsymbol{\mu}_i)_{i\geq 1}$ are the points of a Poisson point process on $(0,\infty )\times \mathcal{T} \times \mathds{S}^{2}$ with intensity $u^{-2} \mathrm{d}u\times \lambda(\mathrm{d}b) \times \lambda_{\mathds{S}^2}(\mathrm{d}\boldsymbol{\mu})$.
\end{Def}

Second, the class of processes presented right above can be extended by allowing in \eqref{SpectralRepresentation} any probability density function $\tilde{f}: \mathds{S}^2 \to [0, \infty)$ involving one mean direction parameter in $\mathds{S}^2$, and not only the function $f$ defined in \eqref{Eq_von_Mises_Fisher}. The resulting models constitute the so called generalised subclass of models of type 2. For several examples of probability density functions on spheres, we refer the reader to \cite{mardia2009directional}, Section 9.3. 

\begin{Rq}
The generalised subclass of models of type 2 is included in the class of space-time max-stable models with spectral separability introduced in \cite{embrechts2016space}, Definition 1. For an explanation about the interest of this class of space-time max-stable models compared to those previously introduced in the literature and an interpretation of its different components, we refer the reader to \cite{embrechts2016space}.
\end{Rq}

\section{Markovian models and geometric ergodicity}
\label{Sec_Geometric_Ergodicity}

In this section, if $\mathcal{T}=\mathds{R}$, let $g$ be the density of a standard exponential
random variable whereas if $\mathcal{T}=\mathds{Z}$, let $g$
correspond to the probability mass function of a geometric random variable: 
\begin{equation}
g(t)=\left\{ 
\begin{array}{ll}
\nu \exp (-\nu t)\ \mathds{I}_{\{t\geq 0\}} & \mbox{if }\mathcal{T}=\mathds{R
}, \\ 
(1-\phi )\phi ^{t}\ \mathds{I}_{\{t\geq 0\}} & \mbox{if }\mathcal{T}=\mathds{
Z},
\end{array}
\right.   \label{Eq_Function_g}
\end{equation}
where $\nu >0$ and $\phi \in (0,1)$. Let us denote by $a$ the constant $\exp(-\nu)$ if $\mathcal{T} = \mathds{R}$ and the constant $\phi$ if $\mathcal{T} = \mathds{Z}$. Combining \eqref{SpectralRepresentation} and \eqref{Eq_Function_g}, the models of type 2 described in Section \ref{Sec_Space_Time_Max_Stable_Sphere} become, for $t \in \mathcal{T}$ and $\Mb{x} \in \mathds{S}^2$,
\begin{equation}
X(t,\mathbf{x})=\left\{ 
\begin{array}{ll}
\bigvee_{i=1}^{\infty} \left\{ U_{i}\nu \exp (-\nu (t-B_{i}) ) \mathds{I}_{\{t-B_{i}\geq 0\}} f(R_{\theta(t-B_i), \Mb{u}} \mathbf{x}; \boldsymbol{\mu}_i, \kappa) \right\}  & \mbox{if }\mathcal{T}=\mathds{R}, \\ 
\bigvee_{i=1}^{\infty} \left\{ U_{i}\phi (1-\phi )^{t-B_{i}}\mathds{I}
_{\{t-B_{i}\geq 0\}}f(R_{\theta(t-B_i), \Mb{u}} \mathbf{x}; \boldsymbol{\mu}_i, \kappa)\right\}  & \mbox{if }\mathcal{T}=\mathds{Z},
\end{array}
\right.   \label{Model_Sphere_Markovian}
\end{equation}
where $(U_{i},B_i, \boldsymbol{\mu }_{i})_{i\geq 1}$ are the points of a Poisson point process on $
(0,\infty )\times \mathcal{T} \times \mathds{S}^{2}$ with intensity $u^{-2}\mathrm{d}u\times \lambda(\mathrm{d}b) \times \lambda _{\mathds{S}^2}(\mathrm{d}\boldsymbol{\mu})$ and $f$ is given by \eqref{Eq_von_Mises_Fisher}. Similarly as in the case of Markovian models of types 1 and 4 defined in \cite{embrechts2016space}, Section 3.1, we have the following result.
\begin{Th}
\label{Th_Iteration} 
The process $X((t, \Mb{x}))_{(t, \Mb{x}) \in \mathcal{T} \times \mathds{S}^2}$ defined in \eqref{Model_Sphere_Markovian} satisfies, for all $t,s\in \mathcal{T}$ such that $s>0$ and $\Mb{x} \in \mathds{S}^{2}$, 
\begin{equation}
\label{Eq_Iteration}
X(t,\mathbf{x})=\max \left \{ a^{s}X(t-s,R_{\theta s, \Mb{u}}\mathbf{x}),(1-a^{s})Z(t,\mathbf{x}) \right \},  
\end{equation}
where the process $(Z(t,\mathbf{x}))_{\mathbf{x}\in \mathds{S}^{2}}$ is
independent of $(X(t-s,\mathbf{x}))_{\mathbf{x}\in \mathds{S}^{2}}$ and
\begin{equation}
\left( Z(t,\mathbf{x})\right)_{\mathbf{x} \in \mathds{S}^{2}} \overset{d}{=}\left( \bigvee_{i=1}^{\infty } \left \{ U_{i} f(\Mb{x}; \boldsymbol{\mu}_i, \kappa) \right \} \right)_{\mathbf{x} \in \mathds{S}^{2}},  
\label{SpectralrepresentationZ}
\end{equation}
where $(U_{i},\boldsymbol{\mu}_{i})_{i\geq 1}$ are the points of a Poisson point process on $(0,\infty
)\times \mathds{S}^{2}$ with intensity $u^{-2}\mathrm{d}u\times \lambda _{\mathds{S}^2}(\mathrm{d}\boldsymbol{\mu})$. 
\end{Th}

\begin{proof}
The proof is similar as that of Theorem 3, Bullet (i), in \cite{embrechts2016space}. We highlight now the main differences. We consider, for $t,b \in \mathcal{T}$ and $\Mb{x} \in \mathds{S}^2$, $R_{(t, b)} \Mb{x}=R_{\theta(t-b), \Mb{u}} \Mb{x}$, where $\theta \in \mathds{R}$ and $\Mb{u} \in \mathds{S}^2$, instead of $R_{(t,b)} \mathbf{x}=\mathbf{x}-(t-b)\boldsymbol{\tau }$, where $\boldsymbol{\tau }\in \mathds{R}^{2}$. Furthermore, in order to establish \eqref{Eq_Iteration}, we use the fact that, for all $t,s,B_i \in \mathcal{T}$, $\theta \in \mathds{R}$, $\Mb{u} \in \mathds{S}^2$ and $\Mb{x} \in  \mathds{S}^2$, $R_{\theta(s+t-s-B_i), \Mb{u}} \Mb{x}=R_{\theta(t-s-B_i), \Mb{u}} (R_{\theta s, \Mb{u}} \Mb{x})$.
Finally, to prove \eqref{SpectralrepresentationZ}, we take advantage of the fact that, for all $M \in \mathds{N} \backslash \{0\}$, $\Mb{x}_1, \dots, \Mb{x}_M \in \mathds{R}^2$, $t, b \in \mathcal{T}$, $\theta \in \mathds{R}$, $\Mb{u} \in \mathds{S}^2$ and $\kappa \geq 0$, 
\Beq
\label{Eq_Stability_Fisher_vonMises}
\int_{\mathds{S}^2} \bigvee_{m=1}^M \{ f(R_{\theta(t-b), \Mb{u}} \Mb{x}_m; \boldsymbol{\mu}, \kappa) \} \lambda_{\mathds{S}^2}(\mathrm{d} \boldsymbol{\mu}) = \int_{\mathds{S}^2} \bigvee_{m=1}^M \{ f(\Mb{x}_m; \boldsymbol{\mu}, \kappa) \} \lambda_{\mathds{S}^2}(\mathrm{d} \boldsymbol{\mu}).
\Eeq
The latter inequality has been shown in the proof of Theorem 1 in \cite{embrechts2016space}.
\end{proof}

Owing to the results in Theorem \ref{Th_Iteration} (especially \eqref{Eq_Iteration}), we focus in the remainder of the paper on the Markov chain $\left( ( X(t,\mathbf{x}) )_{\mathbf{x}\in \mathds{S}^{2}}\right)_{t\in \mathds{Z}}$ satisfying, for all $t \in \mathds{Z}$ and $\Mb{x} \in \mathds{S}^{2}$, 
\Beq
\label{Eq_Recurrence_Equation_Lag_1}
X(t,\mathbf{x})=\max \{ aX(t-1,R_{\theta ,\mathbf{u}}\mathbf{x}),(1-a)Z(t,\mathbf{x}) \},
\Eeq 
where the $(Z(t,\mathbf{x}))_{\mathbf{x}\in \mathds{S}^{2}}$, $t \in \mathds{Z}$, are independent replications of the process $(Z(\Mb{x}))_{\Mb{x} \in \mathds{S}^2}$ defined by
\begin{equation}
\left( Z(\mathbf{x})\right)_{\mathbf{x} \in \mathds{S}^{2}} = \left( \bigvee_{i=1}^{\infty } \left \{ U_{i} f(\Mb{x}; \boldsymbol{\mu}_i, \kappa) \right \} \right)_{\mathbf{x} \in \mathds{S}^{2}},  
\label{Eq_Definition_Z}
\end{equation}
where $(U_{i},\boldsymbol{\mu}_{i})_{i\geq 1}$ are the points of a Poisson point process on $(0,\infty
)\times \mathds{S}^{2}$ with intensity $u^{-2}\mathrm{d}u\times \lambda_{\mathds{S}^2}(\mathrm{d}\boldsymbol{\mu})$. Now, let $\mathcal{C}_{\mathds{S}^{2}}=\mathcal{C}(\mathds{S}^{2},(0, \infty))$ be the space of continuous functions from $\mathds{S}^{2}$ to $(0, \infty)$ with the topology induced by the uniform metric, i.e. $d(h_{1},h_{2})=\left\Vert h_{1}-h_{2}\right\Vert _{\infty }$, where, for $h\in \mathcal{C}_{\mathds{S}^{2}}$, $\left\Vert h\right\Vert _{\infty}=\sup_{\mathbf{x}\in \mathds{S}^{2}} \{ |h(\mathbf{x})| \}$. Since $Z$ has standard Fr\'echet margins, it is almost surely (a.s.) positive. Moreover, as $f$ is defined on a compact set and is continuous, it is bounded. Hence, using similar arguments as for Theorem 4 in \cite{schlather2002models}, there exists an a.s. finite integer $I$ such that 
$$ \left( Z (\mathbf{x})\right)_{\mathbf{x} \in \mathds{S}^{2}}= \left( \bigvee_{i=1}^{I} \left \{ U_{i} f(\Mb{x}; \boldsymbol{\mu}_i, \kappa) \right \} \right)_{\mathbf{x} \in \mathds{S}^{2}}.$$ 
Accordingly, since $f$ is continuous, we directly obtain that $Z$ is sample-continuous. Hence, it follows from \eqref{Eq_Recurrence_Equation_Lag_1} that the Markov chain $\left( ( X(t,\mathbf{x}) )_{\mathbf{x}\in \mathds{S}^{2}}\right)_{t\in \mathds{Z}}$ takes values in $\mathcal{C}_{\mathds{S}^{2}}$. In addition, $X$ is time-homogeneous since the distribution of the innovation processes $(Z(t,\Mb{x}))_{\mathbf{x}\in \mathds{S}^{2}}$, $t \in \mathds{Z}$, does not depend on $t$.
We have the following result.
\begin{Th}
Let $(Z(t,\mathbf{x}))_{\mathbf{x}\in \mathds{S}^{2}}$, $t \in \mathds{Z}$, be independent replications of the process $(Z(\Mb{x}))_{\Mb{x} \in \mathds{S}^2}$ defined by \eqref{Eq_Definition_Z}. The Markov chain $X$ has a unique invariant probability measure on $\mathcal{C}_{\mathds{S}^{2}}$ which is entirely characterised by the finite-dimensional distributions of the process
\begin{equation}
\label{Eq_Max_Integral_Representation}
\left( \bigvee_{j=0}^{\infty} \left\{ a^{j}(1-a)Z(t-j,R_{\theta j, \Mb{u}} \mathbf{x}) \right\} \right)_{\mathbf{x} \in \mathds{S}^{2}},
\end{equation}
for any $t \in \mathds{Z}$.
\end{Th}

\begin{proof}
Using similar arguments as in the proof of Proposition 2.2 in \cite{davis1989basic}, it is easily shown that \eqref{Eq_Recurrence_Equation_Lag_1} has a unique time-stationary\footnote{In this paper, stationarity refers to strict stationarity.} solution. This yields that the Markov chain $X$ has a unique invariant probability measure. Moreover, using the fact that, for all $\theta \in \mathds{R}$, $\Mb{u} \in \mathds{S}^2$ and $j =0, 1, \dots$, $R_{\theta j, \Mb{u}} R_{\theta, \Mb{u}}=R_{\theta(j+1), \Mb{u}}$, it is readily shown that the process in \eqref{Eq_Max_Integral_Representation} is a time-stationary solution of \eqref{Eq_Recurrence_Equation_Lag_1}. From the uniqueness of the solution, we deduce that the Markov chain $X$ has at each date the same finite-dimensional distributions as \eqref{Eq_Max_Integral_Representation}. Finally, the distribution (in the sense of the induced probability measure on $\mathcal{C}_{\mathds{S}^{2}}$) of $X$ is entirely characterised by its finite-dimensional distributions. This concludes the proof. 
\end{proof}

It is clear that $\mathcal{C}_{\mathds{S}^{2}}$ equipped with the previously defined uniform norm is an infinite-dimensional normed vector space. Hence, Riesz Theorem \citep[see, e.g.,][Section 1.3.14]{aldrovandi2017introduction} immediately gives that this space is not locally compact. Thus, the most classical results about geometric ergodicity of Markov chains, to be found e.g. in \cite{meyn2009markov}, cannot be applied. However, since $\mathds{S}^{2}$ is a compact and metrisable space and $(0, \infty)$ is a Polish space, Theorem 4.19 in \cite{kechris2012classical} gives that $C_{\mathds{S}^{2}}$ is a Polish space. Thus, instead, we use results for Markov chains with Polish state spaces. Such results are uncommon in the literature and here we use those by \cite{hairer2010p}.

Let $L$ be a function from $\mathcal{C}_{\mathds{S}^{2}}$ to $[0,\infty )$
and let us introduce a weighted supremum norm on the space of functions from 
$\mathcal{C}_{\mathds{S}^{2}}$ to $\mathds{R}$ in the following way. For any function $
\varphi :\mathcal{C}_{\mathds{S}^{2}}\mapsto \mathds{R}$, we define
\begin{equation*}
\left\Vert \varphi \right\Vert _{L}=\sup_{h\in \mathcal{C}_{\mathds{S}^{2}}} \left \{
\frac{|\varphi (h)|}{1+L(h)} \right \}.
\end{equation*}
Furthermore, for a probability measure $\eta$ on $\mathcal{C}_{\mathds{S}^{2}}$ and $\varphi :\mathcal{C}_{\mathds{S}^{2}}\mapsto \mathds{R}$, let us denote $\eta(\varphi)=\int_{\mathcal{C}_{\mathds{S}^{2}}}\varphi (y)\eta(\mathrm{d}y)$. We denote by $\mathcal{B}(\mathcal{C}_{\mathds{S}^{2}})$ the Borel $\sigma$-field of $\mathcal{C}_{\mathds{S}^{2}}$. For $h \in \mathcal{C}_{\mathds{S}^{2}}$, $B \in \mathcal{B}(\mathcal{C}_{\mathds{S}^{2}})$ and $n \in \mathds{N} \backslash \{ 0 \}$, let $\mathcal{P}(h,B)$ and $\mathcal{P}^n(h,B)$ be respectively the transition probability and the $n$-step transition probability from $h$ to $B$ associated with the Markov chain $\left( ( X(t,\mathbf{x}) )_{\mathbf{x}\in \mathds{S}^{2}}\right) _{t\in \mathds{Z}}$. Finally, we denote by $\mathcal{P}$ its transition kernel, defined by $\mathcal{P}= \{ \mathcal{P}(h,B), h \in \mathcal{C}_{\mathds{S}^{2}}, B \in \mathcal{B}(\mathcal{C}_{\mathds{S}^{2}}) \}$. Likewise, let $\mathcal{P}^n$ be its $n$-step transition kernel, written as $\mathcal{P}^n= \{ \mathcal{P}^n(h,B), h \in \mathcal{C}_{\mathds{S}^{2}}, B \in \mathcal{B}(\mathcal{C}_{\mathds{S}^{2}}) \}$. As in \cite{hairer2010p}, we also use the notations $\mathcal{P}$ and $\mathcal{P}^n$ for the operators defined, for any measurable function $V: \mathcal{C}_{\mathds{S}^{2}} \to \mathds{R}$ and any $h\in \mathcal{C}_{\mathds{S}^{2}}$, by
\Beq
\label{Eq_Application_Kernel_Function}
(\mathcal{P} V)(h)=\int_{\mathcal{C}_{\mathds{S}^{2}}} V(y) \mathcal{P}(h, \mathrm{d}y) \quad \mbox{and} \quad (\mathcal{P}^n V)(h)=\int_{\mathcal{C}_{\mathds{S}^{2}}} V(y) \mathcal{P}^n(h, \mathrm{d}y).
\Eeq
Our main result states the geometric ergodicity of the Markov chain $\left( ( X(t,\mathbf{x}) )_{\mathbf{x}\in 
\mathds{S}^{2}}\right) _{t\in \mathds{Z}}$. 

\begin{Th}[Geometric ergodicity]
\label{Prop_Geometric_Ergodicity} Let $\pi_{\star}$ be the unique invariant probability measure of the Markov chain $\left( ( X(t,\mathbf{x}) )_{\mathbf{x}\in 
\mathds{S}^{2}}\right) _{t\in \mathds{Z}}$. Furthermore, let, for $h\in \mathcal{C}_{
\mathds{S}^{2}}$, $L(h)=\left\Vert h^{\gamma }\right\Vert _{\infty }$, where $\gamma \in (0,1)$. Then there exist constants $C>0$ and $\rho \in (0,1)$ such that
\begin{equation*}
\left\Vert \mathcal{P}^{n}\varphi -\pi_{\star} (\varphi )\right\Vert _{L} \leq C\rho
^{n}\left\Vert \varphi -\pi_{\star} (\varphi )\right\Vert _{L}
\end{equation*}
holds for every measurable function $\varphi :\mathcal{C}_{\mathds{S}
^{2}}\mapsto \mathds{R}$ such that $\left\Vert \varphi \right\Vert
_{L}<\infty $.
\end{Th}

\begin{proof}
We show that the two assumptions required in Theorem 3.6 in \cite{hairer2010p} are satisfied. 

\medskip 

\noindent \textbf{Assumption 1} \citep[Assumption 3.1 in][]{hairer2010p}\textbf{.}
\textit{There exists a function $L:\mathcal{C}_{\mathds{S}^{2}}\rightarrow \lbrack 0,\infty )$ and constants $K\geq 0$ and $\beta \in (0,1)$ such that
$$
(\mathcal{P}L)(h) \leq \beta L(h)+K,
$$
for all $h\in \mathcal{C}_{\mathds{S}^{2}}$.
} 

\medskip

The function $L$ is here $L(h)=\left\Vert h^{\gamma}\right\Vert _{\infty }, h\in \mathcal{C}_{
\mathds{S}^{2}}$, where $\gamma \in (0,1)$. Using the first part of \eqref{Eq_Application_Kernel_Function} and the fact that $\left( ( X(t,\mathbf{x}) )_{\mathbf{x}\in \mathds{S}^{2}}\right)_{t\in \mathds{Z}}$ is a time-homogeneous Markov chain, we see that, for all $h\in \mathcal{C}_{\mathds{S}^{2}}$ and $t \in \mathds{Z}$,
\Beq
\label{Eq_Expression2_Application_Kernel_Function}
(\mathcal{P}L)(h)= \mathds{E}[L(X(t,\cdot))|X(t-1,\cdot)=h(\cdot)].
\Eeq
Now, observe that, if we denote $\left( \Gamma_{i}\right)_{i\geq 1}=\left( U_{i}^{-1}\right) _{i\geq 1}$, where the $\left( U_{i}\right)_{i\geq 1}$ are as in \eqref{Eq_Definition_Z}, then the $\left( \Gamma_{i}\right) _{i\geq 1}$ are the points of an homogeneous Poisson point process on $(0, \infty)$ with constant intensity equal to one. Hence, the highest $U_i$ corresponds to the smallest $\Gamma_i$, which follows the standard exponential distribution. Thus, its inverse follows the standard Fr\'echet distribution. Moreover, for $f$ defined in \eqref{Eq_von_Mises_Fisher}, $\| f \|_{\infty}$ is reached for $\Mb{x}=\boldsymbol{\mu}$ and is finite. Therefore, the process $Z$ defined in \eqref{Eq_Definition_Z} satisfies
\begin{equation}
\label{Eq_1_Proof_Th4}
\left\Vert Z^{\gamma }\right\Vert _{\infty }\overset{d}{=}Y^{\gamma
}\left\Vert f^{\gamma }\right\Vert _{\infty },
\end{equation}
where $Y$ is a random variable with standard Fr\'{e}chet distribution. Moreover, for all $t \in \mathds{Z}$, 
\begin{eqnarray}
\left\Vert \max \{ ah(R_{\theta, \Mb{u} }\cdot),(1-a)Z(t, \cdot) \}^{\gamma
}\right\Vert _{\infty } &=&\left\Vert \max \{ a^{\gamma }h^{\gamma }(R_{\theta, \Mb{u}}\cdot),(1-a)^{\gamma }Z^{\gamma }(t, \cdot) \}\right\Vert _{\infty } \nonumber \\
&= &\max \{ a^{\gamma }\left\Vert h^{\gamma }\right\Vert _{\infty
},(1-a)^{\gamma }\left\Vert Z^{\gamma }\right\Vert _{\infty } \}.
\label{Eq_2_Proof_Th4}
\end{eqnarray}
Using \eqref{Eq_Recurrence_Equation_Lag_1}, \eqref{Eq_Expression2_Application_Kernel_Function}, \eqref{Eq_1_Proof_Th4}, \eqref{Eq_2_Proof_Th4} and denoting by $\Gamma$ the gamma function, we obtain, for all $h\in \mathcal{C}_{\mathds{S}^{2}}$ and $t \in \mathds{Z}$,
\begin{align*}
(\mathcal{P}L)(h) &=\mathds{E} [\max \{ a^{\gamma }\left\Vert h^{\gamma }\right\Vert
_{\infty },(1-a)^{\gamma }\left\Vert Z^{\gamma }\right\Vert _{\infty } \} ] 
\\&= a^{\gamma }\left\Vert h^{\gamma }\right\Vert _{\infty }\mathds{P} \left(
Y^{\gamma } \leq \frac{ a^{\gamma }\left\Vert h^{\gamma }\right\Vert _{\infty
}} {(1-a)^{\gamma }\left\Vert f^{\gamma } \right\Vert_{\infty} } \right) +(1-a)^{\gamma } \| f^{\gamma} \|_{\infty}
\mathds{E}  \left[ Y^{\gamma }\mathds{I}_{ \left \{ Y^{\gamma }\geq \frac{ a^{\gamma }\left\Vert
h^{\gamma }\right\Vert _{\infty }} { (1-a)^{\gamma }\left\Vert f^{\gamma
}\right\Vert_{\infty} } \right \} } \right] 
\\&\leq a^{\gamma }\left\Vert h^{\gamma }\right\Vert _{\infty }+(1-a)^{\gamma
} \| f^{\gamma} \|_{\infty} \Gamma (1-\gamma ) 
\\& = \beta L(h)+K,
\end{align*}
where $\beta=a^{\gamma } \in (0,1)$ (since $a \in (0,1)$) and $K=(1-a)^{\gamma } \| f^{\gamma} \|_{\infty} \Gamma (1-\gamma ) \geq 0$ (since $\Gamma (1-\gamma )>0$).
Hence, Assumption 1 is satisfied.

\medskip 

\noindent \textbf{Assumption 2} \citep[Assumption 3.4 in][]{hairer2010p}\textbf{.}
\textit{
We denote by $\| . \|_{TV}$ the total variation distance between two probability measures.
For every $R>0$, there exists a constant $\alpha >0$ such that
\Beq
\label{Eq_Assumption2}
\sup_{h_{1},h_{2}\in D_R} \{ \left\Vert \mathcal{P}(h_1,\cdot)-\mathcal{P}(h_2,\cdot) \right\Vert _{TV} \} \leq 2(1-\alpha ),
\Eeq
where $D_R=\{h_{1},h_{2}:L(h_{1})+L(h_{2})\leq R\}$.}

\medskip

Remark 3.5 in \cite{hairer2010p} gives that Condition \eqref{Eq_Assumption2} is equivalent to the fact that
$$ | (\mathcal{P}\varphi)(h_1)-(\mathcal{P}\varphi)(h_2) | \leq 2(1-\alpha )$$
holds uniformly on $\mathcal{G}= \left \{ \varphi: \mathcal{C}_{\mathds{S}^2} \to \mathds{R} : \varphi \mbox{ measurable and } \left\Vert \varphi \right\Vert _{\infty }\leq 1 \right \}$. Consequently, taking advantage of \eqref{Eq_Expression2_Application_Kernel_Function}, we see that we need to prove that, for all $t \in \mathds{Z}$,
$$ \sup_{h_{1},h_{2}\in D_R,\varphi \in \mathcal{G}} \{  \left \vert \mathds{E}\left[ \varphi (X(t,\mathbf{\cdot }))|X(t-1,\mathbf{
\cdot })=h_{1}(\cdot )\right] -\mathds{E}\left[ \varphi (X(t,\mathbf{\cdot }
))|X(t-1,\mathbf{\cdot }) 
=h_{2}(\cdot )\right] \right\vert \} 
\leq 2(1-\alpha ).
$$
Using \eqref{Eq_Recurrence_Equation_Lag_1}, we have, for all $\varphi \in \mathcal{G}$, $h_{1},h_{2}\in D_R$ and $t \in \mathds{Z}$, that
\begin{align}
& \ \ \ \  \left\vert \mathds{E}\left[ \varphi (X(t,\mathbf{\cdot }))|X(t-1,\mathbf{
\cdot })=h_{1}(\cdot )\right] -\mathds{E}\left[ \varphi (X(t,\mathbf{\cdot }
))|X(t-1,\mathbf{\cdot })=h_{2}(\cdot )\right] \right\vert  \nonumber
\\&= \left\vert \mathds{E}\left[ \varphi ( (\max \{ ah_{1}(R_{\theta, \Mb{u} }\mathbf{x}
),(1-a)Z(t, \mathbf{x})\})_{\Mb{x} \in \mathds{S}^2} )-\varphi ( ( \max \{ ah_{2}(R_{\theta, \Mb{u} }\mathbf{x}),(1-a)Z(t, 
\mathbf{x}) \})_{\Mb{x} \in \mathds{S}^2} )\right] \right\vert  \nonumber
\\&\leq \mathds{E}\left[ \left\vert \varphi ( ( \max \{ ah_{1}(R_{\theta, \Mb{u} }\mathbf{x
}),(1-a)Z(t, \mathbf{x})\})_{\Mb{x} \in \mathds{S}^2} )-\varphi ( ( \max \{ ah_{2}(R_{\theta, \Mb{u} }\mathbf{x}),(1-a)Z(t,
\mathbf{x})\})_{\Mb{x} \in \mathds{S}^2})\right\vert \right].
\label{Eq_3_Proof_Th4}
\end{align}
Moreover, for $p,q,z\in \mathcal{C}_{\mathds{S}^{2}}$, it is clear that if, for all $\mathbf{x} \in \mathds{S}^{2}$,
$\max \{ p(\mathbf{x}), q(\mathbf{x}) \} \leq z(\mathbf{x})$,
then, for all $\varphi \in \mathcal{G}$,
$$\left\vert \varphi ( ( \max \{ p(\Mb{x}),z(\Mb{x}) \} )_{\Mb{x} \in \mathds{S}^2})-\varphi ( (\max \{ q(\Mb{x}),z(\Mb{x}) \})_{\Mb{x} \in \mathds{S}^2})\right\vert =0.$$
Therefore, for all $\varphi \in \mathcal{G}$, using the fact that $\| \varphi \|_{\infty} \leq 1$, we have that 
\begin{align}
&\quad \ \mathds{E}\left[ \left\vert \varphi ( (\max \{ ah_{1}(R_{\theta, \Mb{u} }\mathbf{x}
),(1-a)Z(t, \mathbf{x}) \})_{\Mb{x} \in \mathds{S}^2})-\varphi ( (\max \{ ah_{2}(R_{\theta, \Mb{u} }\mathbf{x}),(1-a)Z(t,
\mathbf{x}) \})_{\Mb{x} \in \mathds{S}^2}) \right\vert \right] \nonumber  
\\&\leq 2\left[ 1-\mathds{P}\left( \bigcap_{\Mb{x} \in \mathds{S}^{2}} \left \{ \left \{ Z(t, \mathbf{x})\geq \frac{a}{
(1-a)} \max \{ h_{1}(R_{\theta, \Mb{u} }\mathbf{x}), h_{2}(R_{\theta, \Mb{u} }\mathbf{x}) \} \right \} \right \} \right)
\right]  \nonumber
\\& \leq 2 \left[1 - \mathds{P}\left( \inf_{\mathbf{
x} \in \mathds{S}^{2}} \{ Z(t, \mathbf{x}) \} \geq \frac{a}{(1-a)} \max\{ \Vert h_{1} \Vert _{\infty
}, \left\Vert h_{2}\right\Vert _{\infty }\} \right) \right].
\label{Eq_4_Proof_Th4}
\end{align}
The quantity $\inf_{\mathbf{x}\in \mathds{S}^{2}} \{ f(\Mb{x}; \boldsymbol{\mu}_1, \kappa) \}$ is reached for $\Mb{x}=-\boldsymbol{\mu}_1$. Hence, the process $Z$ defined in \eqref{Eq_Definition_Z} satisfies, for all $\Mb{x} \in \mathds{S}^2$,
$$
Z(\mathbf{x}) \geq \Gamma_{1}^{-1}f(\mathbf{x}; \boldsymbol{\mu }_{1},\kappa ) \geq \Gamma_{1}^{-1} \inf_{\mathbf{x}\in \mathds{S}^{2}} \{ f(\mathbf{x};\boldsymbol{\mu }
_{1},\kappa ) \} =\Gamma_{1}^{-1}\frac{\kappa }{4\pi \sinh(\kappa) }\exp \left(
-\kappa \right),
$$
with $\Gamma_1=U_1^{-1}$, where $U_1$ appears in the definition of $Z$.
Since the $(Z(t,\mathbf{x}))_{\mathbf{x}\in \mathds{S}^{2}}$, $t \in \mathds{Z}$, are independent replications of the process $(Z(\Mb{x}))_{\Mb{x} \in \mathds{S}^2}$, it follows, noting that $\max \{ \left\Vert h_{1}\right\Vert _{\infty }, \left\Vert
h_{2}\right\Vert _{\infty } \} \leq R$, that, for all $t \in \mathds{Z}$,
\begin{align}
&\quad \ \mathds{P}\left( \inf_{\mathbf{x} \in \mathds{S}^{2}} \{ Z(t, \mathbf{x}) \} \geq \frac{a}{(1-a)}
\max \{ \left \Vert h_{1}\right\Vert _{\infty }, \left\Vert h_{2}\right\Vert
_{\infty } \} \right) \nonumber
\\&\geq \mathds{P}\left( \Gamma_{1}^{-1}\frac{\kappa }{4\pi \sinh(\kappa) }\exp
\left( -\kappa \right) \geq \frac{a}{(1-a)} \max \left \{ \left\Vert h_{1}\right\Vert
_{\infty }, \left\Vert h_{2}\right\Vert _{\infty } \right \} \right)  \nonumber
\\& \geq \mathds{P}\left( \Gamma_{1}^{-1}\geq \frac{4\pi \sinh(\kappa) a}{\kappa
(1-a)}\exp \left( \kappa \right) R\right).
\label{Eq_5_Proof_Th4}
\end{align}
Therefore, combining \eqref{Eq_3_Proof_Th4}, \eqref{Eq_4_Proof_Th4} and \eqref{Eq_5_Proof_Th4}, we obtain, for all $t \in \mathds{Z}$ that,
\begin{align*}
& \quad \  \sup_{h_{1},h_{2}\in D_R, \varphi \in \mathcal{G}}\left\vert \mathds{E}\left[ \varphi (X(t,\mathbf{\cdot }))|X(t-1,\mathbf{
\cdot })=h_{1}(\cdot )\right] -\mathds{E}\left[ \varphi (X(t,\mathbf{\cdot }
))|X(t-1,\mathbf{\cdot })=h_{2}(\cdot )\right] \right\vert
\\& \leq 2 \left( 1- \mathds{P}\left( \Gamma_{1}^{-1}\geq \frac{4\pi \sinh(\kappa) a}{\kappa
(1-a)}\exp \left( \kappa \right) R\right) \right)=2(1-\alpha),
\end{align*}
denoting $$\alpha=\mathds{P}\left( \Gamma_{1}^{-1}\geq \frac{4\pi \sinh(\kappa) a}{\kappa (1-a)}\exp \left( \kappa \right) R\right)>0.$$
Hence, Assumption 2 holds.

\medskip

Finally, the application of Theorem 3.6 in \cite{hairer2010p} yields the result.
\end{proof}
The geometric ergodicity result of Theorem \ref{Prop_Geometric_Ergodicity} has two strong implications. First, let $x_0 \in \mathcal{C}_{\mathds{S}^{2}}$ be a realisation of the Markov chain $\left( ( X(t,\mathbf{x}) )_{\mathbf{x}\in \mathds{S}^{2}}\right)_{t\in \mathds{Z}}$ at a date $t_0 \in \mathds{Z}$. Geometric ergodicity implies that, for any $B \in \mathcal{B}(\mathcal{C}_{\mathds{S}^{2}})$,  $\lim_{t \to \infty} \mathds{P}( (X(t_0+t, \Mb{x}))_{\mathbf{x}\in \mathds{S}^{2}} \in B | (X(t_0, \Mb{x}))_{\mathbf{x}\in \mathds{S}^{2}}=x_0)=\pi_{\star}(B)$; note that the rate of this convergence is geometric. Hence, geometric ergodicity can be viewed as a loss of memory property. Second, assume that $X$, the Markov chain specified by \eqref{Eq_Recurrence_Equation_Lag_1}, is now defined with $t \in \mathds{N}$ and any initial (at time $t=0$) distribution (probability measure on $\mathcal{C}_{\mathds{S}^{2}}$) which is different from the invariant probability measure $\pi_{\star}$. Geometric ergodicity entails that the distribution of the chain at a given date tends at a geometric rate to $\pi_{\star}$.

It is worth highlighting the fact that the results of this section remain valid for all models belonging to the generalised subclass of models of type 2 as soon as $g$ is as in \eqref{Eq_Function_g} and $\tilde{f}$ is positive, continuous and satisfies the adapted version of \eqref{Eq_Stability_Fisher_vonMises}.

We conclude this section with the following remark, which shows that geometric ergodicity (in the sense of Theorem \ref{Prop_Geometric_Ergodicity}) can be shown for Markov chains defined in a similar way as in \eqref{Eq_Recurrence_Equation_Lag_1} but with a process $Z$ which is not necessarily max-stable.
\begin{Rq}
For $a \in (0,1)$, we now consider, provided it exists, the Markov chain $\left( ( X(t,\mathbf{x}) )_{\mathbf{x}\in \mathds{S}^{2}}\right)_{t\in \mathds{Z}}$ satisfying, for all $t \in \mathds{Z}$ and $\Mb{x} \in \mathds{S}^{2}$,
\Beq
\label{Eq_Definition_Markov_Chain_Last_Remark}
X(t,\mathbf{x})=\max \{ a X(t-1,R_{\theta ,\mathbf{u}}\mathbf{x}),(1-a)Z(t,\mathbf{x}) \},
\Eeq
where the $(Z(t,\mathbf{x}))_{\mathbf{x}\in \mathds{S}^{2}}$, $t \in \mathds{Z}$, are independent replications of a bounded, a.s. positive and sample-continuous stochastic process $(Z(\Mb{x}))_{\Mb{x} \in \mathds{S}^2}$. The same result as in Theorem \ref{Prop_Geometric_Ergodicity} can be obtained for the Markov chain $X$ with similar arguments. Denoting by $\mathcal{K}$ some compact set, the same holds true for, provided it exists, the Markov chain $\left( ( X(t,x) )_{x \in \mathcal{K}}\right)_{t\in \mathds{Z}}$ satisfying, for all $t \in \mathds{Z}$ and $x \in \mathcal{K}$,\footnote{The set $\mathcal{K}$ is not necessarily a subset of $\mathds{R}^d$, whence the notation $x$ instead of $\Mb{x}$.}
$$ X(t,x)=\max \{ a X(t-1, x),(1-a)Z(t,x) \},$$
where the $(Z(t,x))_{x \in \mathcal{K}}$, $t \in \mathds{Z}$, are as in \eqref{Eq_Definition_Markov_Chain_Last_Remark} apart from the replacement of $\mathds{S}^{2}$ with $\mathcal{K}$. We thank the referee for pointing this out.
\end{Rq}

\section{Conclusion}
\label{Sec_Conclusion}

The main result of this paper concerns geometric ergodicity of some Markov chains induced by processes belonging to the class of space-time max-stable models with spectral separability introduced in \cite{embrechts2016space}. Since the associated space state is not locally compact, we could not use the classical approach described, e.g., in \cite{meyn2009markov} and had to apply results for Markov chains with Polish state spaces to be found in \cite{hairer2010p}. Some future challenging work might consist in investigating whether or not the Markov chains described in \cite{embrechts2016space}, Section 3.1, are geometrically ergodic. They take values in the space of continuous functions from $\mathds{R}^2$ to $(0, \infty)$. The latter being non locally compact and not Polish (since non separable), neither the results in \cite{meyn2009markov} nor those from \cite{hairer2010p} can be used in that case.

\section*{Acknowledgements} 
Erwan Koch would like to thank RiskLab at ETH Zurich, the Swiss Finance Institute and the Swiss National Science Foundation (project number 200021\_178824) for financial support. Both authors would like to acknowledge Paul Embrechts for interesting discussions and the referee for insightful suggestions.

\newpage
\bibliographystyle{apalike}
\bibliography{Extension_MARMA_Spatial}

\end{document}